\documentclass[12pt, a4paper, leqno]{amsart}
\usepackage{amsmath, amsthm, amssymb}
\usepackage[top=1truein, bottom=1truein, hmargin={1truein, 1truein}]{geometry}
\usepackage{indentfirst}

\DeclareMathOperator{\Rea}{Re}
 \newtheorem{theorem}{Theorem}[section]
 \newtheorem{cor}[theorem]{Corollary}
 \newtheorem{lemma}[theorem]{Lemma}
 \newtheorem{prop}[theorem]{Proposition}
 \theoremstyle{definition}
 
 \theoremstyle{remark}

 \numberwithin{equation}{section}

\begin{document}
\title[On the sum relation of multiple Hurwitz zeta functions]{On the sum relation of multiple Hurwitz zeta functions}

\author[C.-L.~Chung]{Chan-Liang Chung}

\address{Institute of Mathematics, Academia Sinica, 6F, Astronomy-Mathematics Building, No. 1, Sec. 4, Roosevelt Road, Taipei 10617, Taiwan(R.O.C.)}

\email{andrechung@gate.sinica.edu.tw}


\date{Sep. 5, 2016}
\begin{abstract}
In this paper we shall define a special-valued multiple Hurwitz zeta functions, namely the multiple $t$-values $t(\boldsymbol{\alpha})$ and define similarly the multiple star $t$-values as $t^{\star}(\boldsymbol{\alpha})$. Then we consider the sum of all such multiple (star) $t$-values of fixed depth and weight with even argument and prove that such a sum can be evaluated when the evaluations of $t(\{2m\}^n)$ and $t^{\star}(\{2m\}^n)$ are clear. We give the evaluations of them in terms of the classical Euler numbers through their generating functions.
\end{abstract}
\maketitle
\section{Introduction and Statement of the Main Result}
Let $\boldsymbol{\alpha}=(\alpha_1, \alpha_2, \ldots, \alpha_k)$ be a $k$-tuple positive integer, we define the multiple $t$-values of depth $k$ \cite{SC11,Z15} by
\begin{equation*}
\begin{split}
t(\boldsymbol{\alpha})=t(\alpha_1, \alpha_2, \ldots, \alpha_k)=\sum_{1\leq j_1<j_2<\cdots<j_k} \frac{1}{(2j_1-1)^{\alpha_1}(2j_2-1)^{\alpha_2}\cdots (2j_k-1)^{\alpha_k}},
\end{split}
\end{equation*}
which is equal to the multiple Hurwitz zeta functions $2^{-|\boldsymbol{\alpha}|}\zeta(\alpha_1, \ldots, \alpha_k; -\frac{1}{2}, \ldots, -\frac{1}{2})$ having weight $|\boldsymbol{\alpha}|=\alpha_1+\alpha_2+\cdots+\alpha_k$. Let $(\{a\}^n)$ be the string $(a,a,\ldots, a)$ for any positive integer $a$. It is straightforward that
\[
1+\sum_{n=1}^{\infty}t(\{m\}^n)x^{mn}=\prod_{j=1}^{\infty} \left(1+\frac{x^m}{(2j-1)^m}\right).
\]

Similarly, we can define the multiple star $t$-values of depth $k$ and weight $|\boldsymbol{\alpha}|$ by
\begin{equation*}
\begin{split}
t^{\star}(\alpha_1, \alpha_2, \ldots, \alpha_k)=\sum_{1\leq j_1\leq j_2\leq \cdots\leq j_k}\frac{1}{(2j_1-1)^{\alpha_1}(2j_2-1)^{\alpha_2}\cdots (2j_k-1)^{\alpha_k}}.
\end{split}
\end{equation*}
The only change consists in considering the non-strict inequalities under the summation sign. The generating function of $t^{\star}(\{m\}^n)$ is given by
\[
\prod_{j=1}^{\infty}\left(1-\frac{x^m}{(2j-1)^m}\right)^{-1}.
\]
That is,
\[
1+\sum_{n=1}^{\infty}t^{\star}(\{m\}^n)x^{mn}=\prod_{j=1}^{\infty}\left(1-\frac{x^m}{(2j-1)^m}\right)^{-1}.
\]
In this paper, we consider the sum of all multiple $t$-value of depth $k$ and weight $mn$ with argument $m\geq 2$ as
\begin{equation*}
\begin{split}
T(mn,k)=\sum_{|\boldsymbol{\alpha}|=n}t(m\alpha_1, m\alpha_2, \ldots, m\alpha_k).
\end{split}
\end{equation*}
This is equivalent to 
\[
2^{mn}T(mn,k)=\sum_{|\boldsymbol{\alpha}|=n}\zeta\left(m\alpha_1, m\alpha_2, \ldots, m\alpha_k; -\frac{1}{2}, -\frac{1}{2}, \ldots, -\frac{1}{2}\right),
\]
and we put
\begin{equation*}
\begin{split}
T^{\star}(mn,k)=\sum_{|\boldsymbol{\alpha}|=n}t^{\star}(m\alpha_1, m\alpha_2, \ldots, m\alpha_k).
\end{split}
\end{equation*}
There is a simple connection between the evaluations of $T(mn,k)$ and $T^{\star}(mn,k)$ and it could be done by a combinatorial argument that is essentially the same as the proof of Lemma 1 in \cite{H16}.
\begin{lemma}\label{hof}
For positive integers $k\leq n$ and $m\geq 2$, we have
\begin{equation*}
\begin{split}
T^{\star}(mn,k)=\sum_{r=1}^{k} \binom{n-r}{k-r}T(mn,r).
\end{split}
\end{equation*}
\end{lemma}

Next we prove that the evaluations of $T(mn,k)$ and $T^{\star}(mn,k)$ are based on the evaluations of multiple $t$-value $t(\{m\}^p)$ and star $t$-value $t^{\star}(\{m\}^q)$. 
\begin{theorem}\label{main}
For positive integers $k\leq n$ and $m\geq 2$, we have
\begin{equation*}
\begin{split}
T(mn,k)=\sum_{p=k}^{n} (-1)^{p-k}\binom{p}{k}t(\{m\}^p)t^{\star}(\{m\}^{n-p})
\end{split}
\end{equation*}
and
\begin{equation*}
\begin{split}
T^{\star}(mn,k)=\sum_{q=k}^n (-1)^{n+q}\binom{q}{k}t(\{m\}^{n-q})t^{\star}(\{m\}^q).
\end{split}
\end{equation*}
\end{theorem}

By Theorem \ref{main} and the evaluations of $\zeta(\{{2\}}^n)$ and $\zeta^{\star}(\{2\}^n)$ in terms of the classical Euler numbers given in Section 3 (formulas (\ref{t2}) and (\ref{ts2})), we have for positive integers $k\leq n$,
\begin{equation}\label{T2n}
\begin{split}
&T(2n,k)=\frac{(-1)^{n-k}\pi^{2n}}{4^n (2n)!}\sum_{p=k}^n \binom{2n}{2p}\binom{p}{k}E_{2n-2p};\\
&T^{\star}(2n,k)=\frac{(-1)^n\pi^{2n}}{4^n (2n)!}\sum_{q=k}^n \binom{2n}{2q}\binom{q}{k}E_{2q}.
\end{split}
\end{equation}
On the other hand, by Lemma \ref{hof} we also have
\begin{equation*}
\begin{split}
T^{\star}(2n,k)&=\sum_{r=1}^k\binom{n-r}{k-r}T(2n,r)\\
&=\frac{(-1)^n \pi^{2n}}{4^n (2n)!}\sum_{r=1}^k (-1)^r \binom{n-r}{k-r}\sum_{p=r}^n \binom{2n}{2p}\binom{p}{r}E_{2n-2p}. 
\end{split}
\end{equation*}
Therefore there is an Euler-numbers identity behind the two evaluations of $T^{\star}(2n,k)$:
\[
\sum_{r=1}^k (-1)^r\binom{n-r}{k-r}\sum_{p=r}^n \binom{2n}{2p}\binom{p}{r}E_{2n-2p}=\sum_{q=k}^n\binom{2n}{2q}\binom{q}{k}E_{2q}.
\]
Additionally, we list the evaluations of $T(4n,k)$ and $T^{\star}(4n,k)$ as follows
\begin{equation*}
\begin{split}
&T(4n,k)=\frac{(-1)^n \pi^{4n}}{4^n (4n)!}\sum_{p=0}^{n-k}\frac{(-1)^{p+k}}{4^p}\binom{n-p}{k}\binom{4n}{4p}\sum_{\ell_1+\ell_2=p}(-1)^{\ell_1}\binom{4p}{2\ell_1}E_{2\ell_1}E_{2\ell_2};\\
&T^{\star}(4n,k)=\frac{(-1)^n\pi^{4n}}{4^n(4n)!}\sum_{q=k}^n \frac{(-1)^q}{4^q}\binom{q}{k}\binom{4n}{4q}\sum_{\ell_1+\ell_2=q}(-1)^{\ell_1}\binom{4q}{2\ell_1}E_{2\ell_1}E_{2\ell_2}.
\end{split}
\end{equation*}

\section{Proof of Theorem \ref{main}}
Following \cite{CCE16}, for two real variables $y$ and $z$, we form the infinite product
\[
T_{m}(x;y,z)=\prod_{n=1}^{\infty}\left(1+\frac{yx^m}{(2n-1)^m}\right)\left(1-\frac{zx^m}{(2n-1)^m}\right)^{-1}.
\]
Notice that the right hand side of above product are the product of two generating functions of $y^n t(\{m\}^n)$ and $z^n t^{\star}(\{m\}^n)$, respectively. 
\begin{proof}[Proof of Theorem \ref{main}]
It is easy to see that
\[
T_m(x;y,z)=1+\sum_{n=1}^{\infty}\sum_{p+q=n}y^pz^qt(\{m\}^p)t^{\star}(\{m\}^q)x^{mn},
\] 
here in convention we let $t(\{m\}^0)=t^{\star}(\{m\}^0)=1$.

On the other hand, we obtain
\begin{equation*}
\begin{split}
T_m(x;y,z)=\prod_{n=1}^{\infty}\Big[1+&(y+z)\frac{x^m}{(2n-1)^m}+z(y+z)\frac{x^{2m}}{(2n-1)^{2m}}+\cdots \\
&+z^{k-1}(y+z)\frac{x^{km}}{(2n-1)^{km}}+\cdots\Big],
\end{split}
\end{equation*}
or
\[
T_n(x;y,z)=1+\sum_{n=1}^{\infty}\sum_{r=1}^n z^{n-r}(y+z)^rT(mn,r) x^{mn}.
\]
This implies immediately that
\begin{equation}\label{par}
\begin{split}
\sum_{r=1}^n z^{n-r}(y+z)^r T(mn,r)=\sum_{p+q=n}y^pz^qt(\{m\}^p)t^{\star}(\{m\}^q).
\end{split}
\end{equation}
Applying the differential operator $(\partial^k/ \partial y^k)$ to the both sides of above equation and then setting $y=-1, z=1$ to get
\[
k!T(mn,k)=\sum_{p+q=n}\frac{p!}{(p-k)!}(-1)^{p-k}t(\{m\}^p)t^{\star}(\{m\}^q).
\] 
Hence our first assertion of Theorem \ref{main} follows. 

If we apply the differential operator $(\partial^k/ \partial z^k)$ to the both sides of equation (\ref{par}) and take $y=1, z=-1$ afterwards, then
\[
\sum_{r=1}^k \binom{n-r}{k-r}(-1)^{n-k}T(mn,r)=\sum_{p+q=n}(-1)^{q-k}\binom{q}{k}t(\{m\}^p)t^{\star}(\{m\}^q).
\] 
By Lemma \ref{hof}, we obtain the second assertion.
\end{proof}
Taking values $y=0$ and $z=1$ into equation (\ref{par}) gives the following result.
\begin{cor}
For a pair of positive integers $n,m$ with $m\geq 2$, we have
\[
t^{\star}(\{m\}^n)=T(mn,1)+T(mn,2)+\cdots+T(mn,n).
\]
\end{cor}

\section{Evaluations of $t(\{2m\}^n)$ and $t^{\star}(\{2m\}^n)$}
Theorem \ref{main} says that the formula of $T(2mn,k)$ can be deduced directly from the evaluations of $t(\{2m\}^n)$ and $t^{\star}(\{2m\}^n)$. J.~Zhao \cite{Z15} gave the evaluation of $t(\{2\}^n)$ for any positive integer $n$ as follows
\begin{equation}\label{t2}
\begin{split}
t(\{2\}^n)=\frac{\pi^{2n}}{4^n(2n)!}.
\end{split}
\end{equation}
Then he used the theory of symmetric functions established by M.~Hoffman \cite{H16} to calculate that
\begin{equation}\label{ts2}
\begin{split}
t^{\star}(\{2\}^n)=\frac{(-1)^nE_{2n}\pi^{2n}}{4^n(2n)!},
\end{split}
\end{equation}
and for positive integers $k\leq n$,
\[
T(2n,k)=\frac{(-1)^{n-k}\pi^{2n}}{4^n (2n)!}\sum_{\ell=0}^{n-k}\binom{n-\ell}{k}\binom{2n}{2\ell}E_{2\ell},
\]
where $E_{2n}$ is the $2n$-th Euler number defined by
\[
\sec{x}=\sum_{n=0}^{\infty}\frac{(-1)^nE_{2n}}{(2n)!}x^{2n} \quad \mbox{for} \;\;|x|<\frac{\pi}{2}.
\]
This is equivalent to the formula given by (\ref{T2n}) in Section 1.

According to the parity of $m$, we divide the general evaluations of $t(\{2m\}^n)$ by two cases.
\begin{prop}\label{oddt}
For positive integers $n$ and $m$ with $m\geq 3$ odd, we let $w_m=e^{\frac{2\pi i}{m}}$ and have
\begin{equation*}
\begin{split}
t(\{2m\}^n)=\frac{\pi^{2mn}}{2^{m-1}(2mn)!}\sum_{k=1}^{(m-1)/2}\sum_{0\leq j_1<j_2<\cdots <j_k\leq m-1}(w_m^{j_1}+w_m^{j_2}+\cdots +w_m^{j_k})^{2mn}.
\end{split}
\end{equation*}
\end{prop}
\begin{proof}
For $x\neq 0$ and $|x|<1$, we have
\begin{equation*}
\begin{split}
1+\sum_{n=1}^{\infty}(-1)^n t(\{2m\}^n)x^{2mn}&=\prod_{j=1}^{\infty}\left(1-\frac{x^{2m}}{(2j-1)^{2m}}\right)\\
&=\prod_{j=1}^{\infty}\left(1-\frac{x^{2m}}{j^{2m}}\right) \Big/ \left(1-\frac{x^{2m}}{(2j)^{2m}}\right).
\end{split}
\end{equation*}
Note that
\begin{equation*}
\begin{split}
\prod_{j=1}^{\infty}\left(1-\frac{x^{2m}}{j^{2m}}\right)&=\prod_{j=1}^{\infty}\prod_{k=0}^{m-1} \left(1-\frac{(w_m^k x)^2}{j^2}\right)\\
&=\prod_{k=0}^{m-1} \frac{\sin{(w_m^k \pi x)}}{w_m^k \pi x}\\
&=\frac{1}{(\pi x)^{m}}\prod_{k=0}^{m-1} \sin{(w_m^k \pi x)}.
\end{split}
\end{equation*}
Let $y=\pi x/2$. Thus,
\[
\prod_{j=1}^{\infty}\left(1-\frac{x^{2m}}{j^{2m}}\right) \Big/ \left(1-\frac{x^{2m}}{(2j)^{2m}}\right)=\prod_{k=0}^{m-1} \cos{(w_m^k y)}.
\]
We express the product of cosine functions into a linear combination of cosine functions as
\[
\frac{1}{2^{m-1}}\sum_{\varepsilon_j=\pm 1,\; 1\leq j\leq m-1}\cos{(w_m^{m-1}+\varepsilon_1w_m^{m-2}+\cdots+\varepsilon_{m-1})y}.
\]
It can be rewritten as
\[
\frac{1}{2^{m-1}}\sum_{k=1}^{(m-1)/2}\sum_{0\leq j_1<j_2<\cdots< j_k\leq m-1}\cos{(2y(w_m^{j_1}+w_m^{j_2}+\cdots+w_m^{j_k}))},
\]
or
\[
\frac{1}{2^{m-1}}\sum_{k=1}^{(m-1)/2}\sum_{0\leq j_1<j_2<\cdots< j_k\leq m-1}\cos{(\pi x(w_m^{j_1}+w_m^{j_2}+\cdots+w_m^{j_k}))},
\]
since $w_m^m=1$ and for $m\geq 3$ that $w_m^{m-1}+w_m^{m-2}+\cdots+w_m+1=0$. Extracting the coefficient of $x^{2mn}$ from the expression leads to the evaluation of $t(\{2m\}^n)$.
\end{proof}

\begin{prop}\label{event}
For positive integers $n$ and $m$ with $m\geq 2$ even, we let $w=w_{2m}=e^{\frac{2\pi i}{2m}}$ and have
\begin{equation*}
\begin{split}
t(\{2m\}^n)=\frac{(-1)^n \pi^{2mn}}{2^{2mn+m-2}(2mn)!}\Rea{\left(\sum_{\boldsymbol{\varepsilon} \in A}(w^{m-1}+\varepsilon_1w^{m-2}+\cdots+\varepsilon_{m-1})^{2mn}\right)},
\end{split}
\end{equation*}
where $\boldsymbol{\varepsilon}=(\varepsilon_1, \varepsilon_2, \ldots, \varepsilon_{m-1})$ and for each $1\leq j\leq m-1$ we have either $\varepsilon_{j}=1$ or $\varepsilon_j=-1$. Here $A$ is the set of elements of the form $w^{m-1}+\varepsilon_1w^{m-2}+\cdots+\varepsilon_{m-1}$ such that the number of $-1$ in $\boldsymbol{\varepsilon}$ is even.
\end{prop}
\begin{proof}
As in the proof of Proposition \ref{oddt}, we have
\begin{equation*}
\begin{split}
1+\sum_{n=1}^{\infty}(-1)^nt(\{2m\}^n)x^{2mn}&=\prod_{j=1}^{\infty}\left(1-\frac{x^{2m}}{(2j-1)^{2m}}\right)\\
&=\prod_{k=0}^{m-1} \cos{(w^ky)},
\end{split}
\end{equation*}
where $y=\pi x/2$. Now we express the product of cosine functions into a linear combination of cosine functions 
\begin{equation*}
\begin{split}
\prod_{k=0}^{m-1} \cos{(w^k y)}=\frac{1}{2^{m-1}}\sum_{\varepsilon_j=\pm 1,\; 1\leq j\leq m-1}\cos{(w^{m-1}+\varepsilon_1w^{m-2}+\cdots+\varepsilon_{m-1})y}.
\end{split}
\end{equation*}
It immediately follows that
\begin{equation}\label{eq2}
\begin{split}
t(\{2m\}^n)=\frac{(-1)^n}{2^{m-1}}\sum_{\varepsilon_j=\pm 1,\; 1\leq j\leq m-1}\frac{(-1)^{mn}\pi^{2mn}}{(2mn)!2^{2mn}}(w^{m-1}+\varepsilon_1w^{m-2}+\cdots+\varepsilon_{m-1})^{2mn}.
\end{split}
\end{equation}
For any $B_j=w^{m-1}+\varepsilon_1 w^{m-2}+\cdots +\varepsilon_{m-1} \in A^c$, where $A^c$ denote the complement of the set $A$. If $\varepsilon_{m-2}=1$, then
\begin{equation*}
\begin{split}
-\overline{B_j}&=w^{\frac{m}{2}}\left(w^{\frac{m}{2}-1}+\varepsilon_{m-3}w^{\frac{m}{2}-2}+\cdots+\varepsilon_{1}w^{2-\frac{m}{2}}+w^{1-\frac{m}{2}}+i\varepsilon_{m-1}\right)\\
&=w^{m-1}+\varepsilon_{m-3}w^{m-2}+\cdots+\varepsilon_1w^2+w-\varepsilon_{m-1}.
\end{split}
\end{equation*}
It implies $-\overline{B_j}\in A$ since the number of $-1$ in $\varepsilon_{m-3}, \varepsilon_{m-4}, \ldots, \varepsilon_1, 1, -\varepsilon_{m-1}$ is even. 

If $\varepsilon_{m-2}=-1$, then 
\begin{equation*}
\begin{split}
\overline{B_j}&=-w^{\frac{m}{2}}\left(-w^{\frac{m}{2}-1}+\varepsilon_{m-3}w^{\frac{m}{2}-2}+\cdots+\varepsilon_{1}w^{2-\frac{m}{2}}+w^{1-\frac{m}{2}}+i\varepsilon_{m-1}\right)\\
&=w^{m-1}-\varepsilon_{m-3}w^{m-2}-\cdots-\varepsilon_1w^2-w+\varepsilon_{m-1}.
\end{split}
\end{equation*}
Since the number of $-1$ in $-\varepsilon_{m-3},-\varepsilon_{m-4},\ldots, -\varepsilon_1,-1,\varepsilon_{m-1}$ is even, we have $\overline{B_j}\in A$. Thus, there is a one-to-one corresponding from $A^c$ to $A$. Let $A=\{A_1, A_2, \ldots, A_{2^{m-2}}\}$ and $A^c=\{B_1, B_2, \ldots, B_{2^{m-2}}\}$, then 
\begin{equation*}
\begin{split}
&\sum_{\varepsilon_j=\pm 1,\; 1\leq j\leq m-1}(w^{m-1}+\varepsilon_1w^{m-2}+\cdots+\varepsilon_{m-1})^{2mn}\\
&=\;\sum_{j=1}^{2^{m-2}}(A_j^{2mn}+B_j^{2mn})\\
&=\;\sum_{j=1}^{2^{m-2}}(A_j^{2mn}+\overline{A_j}^{2mn})\\
&=\;2\Rea{\left(\sum_{\boldsymbol{\varepsilon} \in A}(w^{m-1}+\varepsilon_1w^{m-2}+\cdots+\varepsilon_{m-1})^{2mn}\right)}.
\end{split}
\end{equation*}
From which and (\ref{eq2}) our assertion follows.
\end{proof}
By Proposition \ref{oddt} and \ref{event}, it immediately follows for any positive integer $n$ that
\begin{equation*}
\begin{split}
&t(\{4\}^n)=\frac{\pi^{4n}}{4^n(4n)!}, \; t(\{6\}^n)=\frac{3\pi^{6n}}{4\cdot(6n)!}\quad \mbox{and}\\
&t(\{8\}^n)=\frac{\pi^{8n}}{2\cdot(8n)!}\left[\left(1+\frac{1}{\sqrt{2}}\right)^{4n}+\left(1-\frac{1}{\sqrt{2}}\right)^{4n}\right].
\end{split}
\end{equation*}

There is also a slight difference when $m$ is even or odd in the general formula of $t^{\star}(\{2m\}^n)$.
\begin{prop}\label{oddts}
For positive integers $n,m$ with $m$ odd, we have
\begin{equation*}
\begin{split}
t^{\star}(\{2m\}^n)=\frac{(-1)^n\pi^{2mn}}{2^{2mn}}\sum_{|\boldsymbol{\ell}|=mn}\prod_{j=0}^{m-1} \frac{E_{2\ell_j}}{(2\ell_j)!}w_m^{2j\ell_{j+1}},
\end{split}
\end{equation*}
where $w_m=e^{\frac{2\pi i}{m}}$ and the summation ranges over all nonnegative integers $\ell_1, \ell_2, \ldots, \ell_m$ such that $\ell_1+\ell_2+\cdots+\ell_m=mn$. 
\end{prop}
\begin{proof}
It is straightforward that
\begin{equation*}
\begin{split}
1+\sum_{n=1}^{\infty}t^{\star}(\{2m\}^n)x^{2mn}&=\prod_{j=1}^{\infty}\left(1-\frac{x^{2m}}{(2j-1)^{2m}}\right)^{-1}\\
&=\prod_{j=1}^{\infty}\left(1-\frac{x^{2m}}{j^{2m}}\right)^{-1} \Big/ \left(1-\frac{x^{2m}}{(2j)^{2m}}\right)^{-1}.
\end{split}
\end{equation*}
Let $w_m=e^{\frac{2\pi i}{m}}$. Note that
\begin{equation*}
\begin{split}
\prod_{j=1}^{\infty}\left(1-\frac{x^{2m}}{j^{2m}}\right)^{-1}&=\prod_{j=1}^{\infty}\prod_{k=0}^{m-1}\left(1-\frac{(w_m^k x)^2}{j^2}\right)^{-1}\\
&=w_m^{\frac{m(m-1)}{2}}(\pi x)^m \prod_{k=0}^{m-1}\csc{(w_m^k \pi x)}.
\end{split}
\end{equation*}
Thus we have
\[
1+\sum_{n=1}^{\infty}t^{\star}(\{2m\}^n)x^{2mn}=\prod_{k=0}^{m-1}\sec{\left(\frac{w_m^k \pi x}{2}\right)}.
\]
Comparing the coefficient of $x^{2mn}$ of the above equation gives the desired evaluation of $t^{\star}(\{2m\}^n)$ for $m$ is odd. 
\end{proof}
Proposition \ref{oddts} implies that, in particular when $m=1$, the formula (\ref{ts2}). In addition, we have
\[
t^{\star}(\{6\}^n)=\frac{(-1)^n\pi^{6n}}{2^{6n}}\sum_{|\boldsymbol{\ell}|=3n}\frac{E_{2\ell_1}E_{2\ell_2}E_{2\ell_3}}{(2\ell_1)!(2\ell_2)!(2\ell_3)!}\left(-\frac{1}{2}+\frac{\sqrt{3}}{2}i\right)^{2\ell_2+4\ell_3}.
\]

\begin{prop}\label{events}
Let $w=e^{\frac{2\pi i}{2m}}$. For positive integers $n,m$ with $m$ even, we have
\begin{equation*}
\begin{split}
t^{\star}(\{2m\}^n)=\frac{\pi^{2mn}}{2^{2mn}}\sum_{|\boldsymbol{\ell}|=mn}\prod_{j=0}^{m-1} \frac{E_{2\ell_j}}{(2\ell_j)!}w^{2j\ell_{j+1}}.
\end{split}
\end{equation*}
\end{prop}
\begin{proof}
As in the proof of Theorem \ref{oddts}, we have
\[
1+\sum_{n=1}^{\infty}t^{\star}(\{2m\}^n)x^{2mn}=\prod_{k=0}^{m-1}\sec{\left(\frac{w^k \pi x}{2}\right)},
\]
where $w=e^{\frac{2\pi i}{2m}}$ in this case. From which we extract the coefficient of $x^{2mn}$ to obtain the desired evaluation of $t^{\star}(\{2m\}^n)$ for even $m$. 
\end{proof}
For example, we have
\begin{equation*}
\begin{split}
&t^{\star}(\{4\}^n)=\frac{\pi^{4n}}{2^{4n}(4n)!}\sum_{\ell=0}^{2n}(-1)^{\ell}\binom{4n}{2\ell}E_{2\ell}E_{4n-2\ell};\\
&t^{\star}(\{8\}^n)=\frac{\pi^{8n}}{2^{8n}}\sum_{|\boldsymbol{\ell}|=4n}\frac{E_{2\ell_1}E_{2\ell_2}E_{2\ell_3}E_{2\ell_4}}{(2\ell_1)!(2\ell_2)!(2\ell_3)!(2\ell_4)!}i^{\ell_2+2\ell_3+3\ell_4}.
\end{split}
\end{equation*}

\providecommand{\bysame}{\leavevmode \hbox to3em%
{\hrulefill}\thinspace}

\end{document}